\documentclass[11pt]{amsart}
\usepackage{amsmath,amsfonts,amssymb,amsthm}
\usepackage[alphabetic]{amsrefs}
\usepackage{hyperref}
\usepackage{graphicx,color}
\usepackage{enumerate}
\usepackage{tikz}

\newtheorem{theorem}{Theorem}[section]
\newtheorem{lemma}[theorem]{Lemma}
\newtheorem{prop}[theorem]{Proposition}
\newtheorem{cor}[theorem]{Corollary}

\newtheorem*{metaconjectureA}{Meta-Conjecture}

\newtheorem*{conjecture}{Conjecture}

\newtheorem{main}{Theorem}

\newtheorem{cmain}[main]{Corollary}
\newtheorem{qmain}[main]{Question}
\newtheorem{conjmain}[main]{Conjecture}

\theoremstyle{definition}
\newtheorem{definition}[theorem]{Definition}

\newtheorem*{question}{Questions}

\newcommand{\N}{\mathbb{N}}
\newcommand{\Z}{\mathbb{Z}}

\newcommand{\catz}{CAT$(0)$ }
\newcommand{\cftf}{$C(4)$-$T(4)$ }

\newcommand{\fix}{Fix}

\newcommand{\NPC}{\mathcal X}

\newcommand{\mk}{\medskip}
\newcommand{\bthm}{\begin{theorem}}
\newcommand{\ethm}{\end{theorem}}
\newcommand{\bmthm}{\begin{main}}
\newcommand{\emthm}{\end{main}}
\newcommand{\ben}{\begin{enumerate}}
\newcommand{\een}{\end{enumerate}}
\newcommand{\bit}{\begin{itemize}}
\newcommand{\eit}{\end{itemize}}
\newcommand{\f}{\frac}
\newcommand{\bp}{\begin{proof} \upshape}
\newcommand{\ep}{\end{proof}}
\newcommand{\beq}{\begin{eqnarray*}}
\newcommand{\eeq}{\end{eqnarray*}}
\newcommand{\ra}{\rightarrow}
\newcommand{\blem}{\begin{lemma}}
\newcommand{\elem}{\end{lemma}}
\newcommand{\bpro}{\begin{prop}}
\newcommand{\epro}{\end{prop}}
\newcommand{\bs}{\symbol{92}}
\newcommand{\bcor}{\begin{cor}}
\newcommand{\ecor}{\end{cor}}
\newcommand{\diam}{\operatorname{diam}}

\renewcommand{\>}{\rangle}
\newcommand{\bmcor}{\begin{cmain}}
\newcommand{\emcor}{\end{cmain}}
\newcommand{\bq}{\begin{question}}
\newcommand{\eq}{\end{question}}
\newcommand{\bmq}{\begin{qmain}}
\newcommand{\emq}{\end{qmain}}
\newcommand{\bmconj}{\begin{conjmain}}
\newcommand{\emconj}{\end{conjmain}}

\newcommand{\st}{\, | \,}

\begin{document}
	
	\title[Locally elliptic actions, torsion groups, and NPC spaces]{Locally elliptic actions, torsion groups, and nonpositively curved spaces}
	
	
	\author[T.~Haettel]{Thomas Haettel$^{\diamondsuit}$}
	\address{IMAG, Univ Montpellier, CNRS, France \\
	Place Eugène Bataillon\\
	34090 Montpellier, France}
	\email{thomas.haettel@umontpellier.fr}
	\thanks{$\diamondsuit$ Partially supported by the French project ANR-16-CE40-0022-01 AGIRA}
	
	\author[D.~Osajda]{Damian Osajda$^{\dag}$}
	\address{Faculty of Mathematics and Computer Science,
		University of Wroc\l aw\\
		pl.\ Grun\-wal\-dzki 2/4,
		50--384 Wroc\-{\l}aw, Poland}
	\address{Institute of Mathematics, Polish Academy of Sciences\\
		\'Sniadeckich 8, 00-656 War\-sza\-wa, Poland}
	\email{dosaj@math.uni.wroc.pl}
	\thanks{$\dag$ Partially supported by (Polish) Narodowe Centrum Nauki, UMO-2018/31/G/ST1/02681.}
	
	\date{\today}

	\begin{abstract}


		Extending and unifying a number of well-known conjectures and open questions, we conjecture that locally elliptic actions (that is, every element has a bounded orbit) by automorphisms of finitely generated groups on finite dimensional nonpositively
		curved complexes have global fixed points. In particular, finitely generated torsion groups cannot act
		without fixed points on such spaces. We prove these conjectures for a wide class of complexes, including all infinite families of Euclidean buildings of classical type, Helly complexes, some graphical small cancellation and systolic complexes, uniformly locally finite Gromov hyperbolic graphs. We present consequences of these results, e.g.\ concerning the automatic continuity.
				
		Our main tool are Helly graphs. We present and study a new notion of geodesic clique paths. Their local-to-global properties are crucial in our proof of ellipticity results.
		
	\end{abstract}
	
	\maketitle
	
	\section{Introduction}
	
	%
	%
	%
	%
	
	\subsection{(Locally) elliptic actions}
	\label{s:metaconj}
	We consider actions of groups on complexes (e.g.\ simplicial complexes or graphs) by automorphisms.  
	Below we formulate a meta-conjecture concerning such actions, being a far-going unification and generalization
	of several well-known conjectures and open questions.
	
	A \emph{locally elliptic action} is an action in which the orbit of each element is bounded. If the orbit of the whole group is bounded we say that the action is \emph{elliptic}.
	The \emph{nonpositive curvature} assumption below refers to a wide understanding of various `nonpositive curvature' conditions as discussed
	in Section~\ref{s:motiv} -- therefore we name the statement below `meta-conjecture' rather than an actual conjecture. Similarly,
	\emph{finite dimensionality} refers to any reasonable notion of dimension -- it might be thought of as the maximum of (topological) dimensions of cells.	
	
	\begin{metaconjectureA}
		Every locally elliptic action of a finitely generated group on a finite-dimensional nonpositively curved complex is elliptic. In particular, every action of a finitely generated torsion group
		on such complex is elliptic.
	\end{metaconjectureA}

Note that Serre proved that every infinitely generated group has an action on a tree without a fixed point (\cite[Theorem~15]{Serre_Trees}), so the assumption of finite generation is crucial. The assumption on finiteness of the dimension is also essential -- Grigorchuk's infinite torsion groups as well as infinite Burnside groups (torsion groups of bounded exponent) can act with unbounded orbits on Hilbert spaces or infinite dimensional \catz cubical complexes \cite{Grigorchuk1984,OsajdaDuke}. Furthermore, the fact that we consider the combinatorial setting of complexes, rather than the one of (arbitrary) metric spaces is important as well -- see Section~\ref{s:motiv} for explanations.

\mk

	For nonpositively curved complexes we have in mind (see Section~\ref{s:motiv}) having a bounded orbit (for a group, or a single automorphism) is usually equivalent to having a fixed point. Therefore, `locally elliptic' can be thought of as `every group element fixes a point' and, correspondingly, `elliptic' should mean `having a global fixed point'. 
	An actual conjecture being a clarification of the above Meta-Conjecture is formulated in Section~\ref{s:motiv}, where
	we also describe in detail motivations and earlier related results. Let us just mention here that the Meta-Conjecture is related to and motivated by, notably: the Tits Alternative, property (FA), Kazhdan's property (T), automatic continuity.
	\medskip
	
	Our main aim in this paper is proving the Meta-Conjecture for a large class of nonpositively curved complexes (see Section~\ref{sec:Helly} for the definition of combinatorial dimension).
	
\bmthm \label{thm:main_elliptic}
		Let a finitely generated group act on a graph $\Gamma$ of finite combinatorial dimension.
		If the action is locally elliptic then it is elliptic. In particular, every action of a finitely generated torsion group on $\Gamma$ is elliptic.
	\emthm
	

Note that in the case when $X$ is a \catz complex, the answer is known only in dimension $2$ (see~\cite{NOP-D}) or with strong restrictions on the group (see~\cite{Izeki_Karlsson}). It is therefore striking that for all graphs with finite combinatorial dimension, we have a very simple and general answer to that question.

\mk

We gather few consequences of Theorem~\ref{thm:main_elliptic} in the following (see also Theorem~\ref{thm:helly_elliptic} in the next section). For more explanations and the description of related results see Section~\ref{s:motiv}.
	
	\bmcor
	\label{cor:fixcor}
		Let a finitely generated group act on a complex $X$ which is one of the following:
		\begin{enumerate}
			\item a Euclidean building of type $\widetilde{A}_n$, $\widetilde{B}_n$, $\widetilde{C}_n$, or $\widetilde{D}_n$;
			\item \label{cor:fixcor2}a uniformly locally finite Gromov hyperbolic graph;
			\item a finite-dimensional CAT(0) cube complex;
			\item \label{cor:fixcor3} a simply connected graphical \cftf small cancellation complex;
			\item \label{cor:fixcor4} a quadric complex;
			\item the Salvetti complex of an FC-type type Artin group.
			\item the Deligne complex of a cyclic type Artin group (see~\cite{haettel_huang_garside_artin_product_Z} for the definition).
		\end{enumerate}
		If the action is locally elliptic then it is elliptic. In particular, every action of a finitely generated torsion group on $X$ is elliptic.
	\emcor

In the case of CAT(0) cube complexes, we recover~\cite[Corollary~B]{leder_varghese}. In the case of Euclidean buildings of type $\tilde{A}$, $\tilde{B}$, $\tilde{C}$ or $\tilde{D}$, we prove a conjecture by Marquis (see~\cite[Conjecture~2]{Marquis2013}). It is also a generalization of~\cite[Corollaire~3]{Parreau2003} and \cite[Corollary~1.2]{NOP-D} in the case of discrete Euclidean buildings. Note that, if $X$ is any building and the action of the group is proper and cocompact, the result follows from~\cite{karpinski_osajda_przytycki}. One should also point out that Breuillard and Fujiwara ask for quantitative versions of these results, see~\cite[Theorem~7.16]{breuillard_fujiwara}.

\mk
		
	The above results indicate the importance of the notion of combinatorial dimension (introduced by Dress in \cite{Dress1984}) for exploring fixed-point questions. Unlike many other notions of dimension -- e.g.\ (virtual) cohomological dimension, geometric dimension, or asymptotic dimension -- the combinatorial dimension has not been yet  studied thoroughly in the context of groups. We believe that the following question is of utmost importance.
	
	\bmq
			Which (nonpositively curved) spaces have finite combinatorial dimension?
	\emq	

	In particular, we pose the following conjecture. Establishing it would extend Corollary~\ref{cor:fixcor}(1) to all (not only Euclidean or hyperbolic) buildings. 
	
	\bmconj\label{conj:buildings}
		Buildings have finite combinatorial dimension. 
	\emconj
	
	Aside from Theorem~\ref{thm:main_elliptic} we provide the following exemplary result confirming the Meta-Conjecture (see Section~\ref{s:relres} for further examples from the literature).
	
	\bmthm \label{thm:systolic_elliptic}
Let $G$ be a finitely generated group acting strongly rigidly by locally elliptic automorphisms on an $18$-systolic complex
$X$.  Then $G$ acts elliptically on $X$.
\emthm	

Although $18$-systolic complexes are (Gromov) hyperbolic, observe that in the above theorem we do not assume that the dimension is finite. 
We believe that the assumptions on high systolicity can be weakened and the assumptions on the strong rigidity of actions can be omitted (see Conjecture in Section~\ref{s:conj}). Still, even in this restricted form Theorem~\ref{thm:systolic_elliptic} applies to an important class of monster groups as constructed in \cite{OsajdaGAFA,Osajda-small}. Note that such groups are not hyperbolic and that  groups with torsion 
constructed there contain elements of arbitrarily large order.

\bmthm \label{thm:C(13)_elliptic}
Finitely generated torsion subgroups of $C(18)$ graphical small cancellation groups are finite.
\emthm

	\subsection{Locally elliptic actions on Helly graphs}
	\label{s:Helly_intro}
	
	In fact, the main object of interest throughout the article and our main tool for proving Theorem~\ref{thm:main_elliptic} are Helly graphs and their automorphism groups.
	
	\mk
		
Helly graphs form a very natural class of metric spaces with nonpositive curvature features, which are of increasing interest in geometric group theory. Their definition is extremely simple: a connected graph is Helly if any family of pairwise intersecting combinatorial balls has a non-empty global intersection, the so-called Helly property. See~\cite{Helly} for a study of group actions on Helly graphs.

\mk

Any graph embeds in an essentially unique minimal Helly graph, its Helly hull. So the study of Helly graphs is really meaningful when some finiteness conditions are assumed: for instance local finiteness, or finite combinatorial dimension.

\mk

One may think of Helly graphs as a very nice class of nonpositively curved, combinatorially defined spaces. In Section~\ref{sec:Helly} we describe it in more details, providing numerous examples of groups acting nicely on Helly graphs and reminding analogies between \catz and Helly, as well as advantages of the combinatorial Helly approach.

The following theorem, adding to the list 
 in Corollary B, is in fact equivalent to Theorem~\ref{thm:main_elliptic}. (In the text, we conclude Theorem~\ref{thm:main_elliptic} from it.) 
	
	\bmthm \label{thm:helly_elliptic}
		Let $X$ denote a Helly graph with finite combinatorial dimension, and let $G$ denote a finitely generated group acting by automorphisms on $X$. Then either $G$ stabilizes a clique in $X$ or some element of $G$ is hyperbolic, in which case it has infinite order.
	\emthm

Note that, if a Helly graph has a bound on the cardinality of its cliques, then it has finite combinatorial dimension.

\mk

We deduce an immediate consequence regarding torsion subgroups of Helly groups.

	\bmcor
		Let $G$ be a Helly group, that is a group acting geometrically on a Helly graph. Any torsion subgroup of $G$ is finite.
	\emcor

	Let us state here just one corollary of the result. It concerns a notion of automatic continuity, having itself origins e.g.\ in the theory of Polish spaces. It is a direct (but non-obvious) consequence of Theorem~\ref{thm:helly_elliptic} and \cite{keppeler2021automatic}.
	
	\bmcor
		Let $G$ be a Helly group, that is a group acting geometrically on a Helly graph. Any group homomorphism $\varphi \colon L \to G$ from a locally compact group $L$ is continuous.	\emcor

According to Marquis (see~\cite[Conjecture~1]{Marquis2013}), we also deduce the following for Euclidean buildings.

\bmcor
Let $G$ denote a locally compact group such that its group of components $G/G^0$ is compact. Let $\Delta$ denote a Euclidean building of type $\tilde{A}$, $\tilde{B}$, $\tilde{C}$ or $\tilde{D}$.

Then any measurable action of $G$ on $\Delta$ fixes a point in $\Delta$. In other words, $G$ has Property (FB) for the class of Euclidean buildings of type $\tilde{A}$, $\tilde{B}$, $\tilde{C}$ or $\tilde{D}$.
\emcor

\mk

The arguments used in the proof Theorem~\ref{thm:helly_elliptic} rely on a new notion of clique-paths, see Section~\ref{sec:clique_paths} for the precise definition. The clique-paths are somewhat similar to normal cube-paths in CAT(0) cube complexes studied by Niblo and Reeves (see~\cite{nibloreeves:groups}) and to normal clique-paths studied by Chalopin et al. (see~\cite{Helly}), in both cases used in the proof of biautomaticity. They are also related to directed geodesics in systolic complexes (see~\cite{JanuszkiewiczSwiatkowski2006}) and to tight geodesics in the curve complex (see~\cite{bowditch_tight}).

However, already in the case of a generic translation of the square grid $\Z^2$, no such normal path will be invariant. Our clique-paths have the advantage of being flexible enough to include such invariant axes (see Proposition~\ref{haettel_automorphisms_helly} for a precise statement). On the other hand, they are sufficiently rigid that, even though they are locally defined, they ensure that any global path is geodesic. This key local-to-global property can be summed up in the following result, see Theorem~\ref{thm:local_to_global_clique_path} for the precise statement.

\bmthm[Local clique-paths are globally geodesic] \label{thm:loc_cliq-pat}\

Let $X$ be a Helly graph, and assume that a bi-infinite sequence of cliques $(\sigma_n)_{n \in \Z}$ is such that for all $n \in \Z$, the sequence $(\sigma_n,\sigma_{n+1},\sigma_{n+2})$ is a clique-path. Then, for any $n,m \in \Z$, we have $d(\sigma_n,\sigma_m) \geq |n-m|$. 
\emthm

Note that, to our knowledge, such a construction is new even for CAT(0) cube complexes.

\mk

The other main ingredient in the proof of Theorem~\ref{thm:helly_elliptic} is the following result about linear growth of orbits in a general injective metric space.

\bmthm[Linear orbit growth] \label{mthm:linear_orbit_growth}\

Let $X$ denote an injective metric space, and let $G,H$ denote elliptic groups of isometries of $X$. If $x \in X^G$ is such 
that $d(x,X^H)=d(X^G,X^H)=L$, there exists a sequence $(g_n)_{n \geq 1}$ in $(GH)^n$ such that
$$\forall n \geq 1, d(g_n \cdot x,x) = 2nL.$$
\emthm

We believe this result could be of independent interest.

\mk

\textbf{Organization of the article:} 
In the following Section~\ref{s:motiv}, we formulate the Conjecture being a precise version of the Meta-Conjecture, we provide motivations and list related results.
The equivalence of Theorems~\ref{thm:locally_elliptic} and \ref{thm:helly_elliptic} is shown in Section~\ref{sec:Helly}. There we also remind basics on Helly graphs and show that Theorem~\ref{thm:helly_elliptic} implies Corollary~\ref{cor:fixcor}.. 
In Section~\ref{sec:Helly}, we review Helly graphs, elliptic automorphism groups and subdivisions of Helly graphs. In Section~\ref{sec:clique_paths}, we define the clique-paths between pairs of cliques in a Helly graph, and prove that local clique-paths are globally geodesic. In Section~\ref{sec:linear_orbit_growth}, we prove the linear growth of orbits in injective metric spaces. In Section~\ref{sec:locally_elliptic}, we use the clique-paths and the linear growth to study locally elliptic actions on Helly graphs and we prove Theorem~\ref{thm:helly_elliptic}. In Section~\ref{s:sc}, we prove Theorem~\ref{thm:systolic_elliptic} about locally elliptic actions on systolic complexes and conclude Theorem~\ref{thm:C(13)_elliptic}. 

\mk

\textbf{Acknowledgments:} 
We would like to thank Jack Button, Victor Chepoi, Sami Douba, Fran\c{c}ois Fillastre, Elia Fioravanti, Anthony Genevois, Hiroshi Hirai, Nima Hoda, Beno{\^ i}t Kloeckner, Vladimir Kovalchuk, Linus Kramer, Urs Lang, Harry Petyt, Betrand R\'{e}my, Olga Varghese and Constantin Vernicos for interesting discussions and remarks on earlier versions of the article. We also thank anonymous referees which helped improve the presentation of the article.

\tableofcontents	

	\section{Conjectures, motivations, and state of the art}
\label{s:motiv}

\subsection{Conjectures}
\label{s:conj}

We now formulate an actual conjecture, being a specification of the Meta-Conjecture. In order to do this we must specify the `nonpositive curvature', and `finite dimensional' assumptions. In what follows, by `dimension'
we mean the maximum  of (topological) dimensions of cells.

\begin{definition}[Finite dimensional NPC spaces]
	\label{d:NPC}
	The class $\NPC$ consists of: 
	\begin{enumerate}
		\item \label{d:NPC2} finite dimensional \catz complexes, cf.\ e.g.\ \cite{BriHaf1999};
		\item \label{d:NPC3} finite dimensional complexes with a convex geodesic bicombing, \cite{Lang2013,DesLang2015};
		\item \label{d:NPC4} finite dimensional complexes with an injective metric \cite{AroPan1956,Isbell64,Dress1984,Lang2013};
		\item \label{d:NPC5} uniformly locally finite Gromov-hyperbolic graphs,
		\item \label{d:NPC6} uniformly locally finite weakly modular graphs \cite{CCHO};
		\item \label{d:NPC7} systolic complexes \cite{Chepoi2000,JanuszkiewiczSwiatkowski2006};
		\item \label{d:NPC8} finite dimensional bucolic complexes \cite{Bresaretal2013};
		\item \label{d:NPC9} finite dimensional weakly systolic complexes \cite{ChepoiOsajda};
		\item \label{d:NPC10} finite dimensional quadric complexes \cite{HodaQuadric};
		\item \label{d:NPC11} graphical $C(6)$, $C(4)$-$T(4)$, and $C(3)$-$T(6)$ small cancellation complexes \cite{OsaPry2018,Helly};
		\item \label{d:NPC12} Salvetti complexes of Artin groups;
		\item \label{d:NPC13} Deligne complexes of Artin groups;
	\end{enumerate}
\end{definition}

\begin{conjecture}
	Let $G$ be a finitely generated group acting by automorphisms on a complex $X\in \NPC$. If the action is locally elliptic then it is elliptic. In particular, if $G$ is a torsion group then the action is elliptic.
\end{conjecture}

\mk

\noindent
{\bf Remarks.}
\noindent
{\bf (A)} In the \catz setting the conjecture concerning the part on torsion groups appears for the first time in 
\cite[Conjecture 1.5]{NOP-D}. Then the \catz version of the Conjecture appears in \cite{Oberwolfach} (in a form of a question). Also in \cite{Oberwolfach} the Helly version of the conjecture appears for the first time. For appearances of related questions and conjectures, see the following Subsection~\ref{s:motiv2}.

\noindent
{\bf (B)} As noted after Meta-Conjecture in Section~\ref{s:metaconj} the assumptions on finite generation and finite dimensionality are essential. In Subsection~\ref{s:motiv2} item (VIII) and in Subsection~\ref{s:relres} we explain why restricting to the combinatorial setting of automorphisms of complexes, rather than (arbitrary) isometries of metric spaces is important.

\noindent
{\bf (C)}	For item (\ref{d:NPC6}) we do not know a general definition of a `weakly modular complex', that is, a contractible complex with $1$-skeleton being a weakly modular graph -- see \cite[Question 9.20]{CCHO}.
	Therefore, we make a (much stronger) assumption of uniform local finiteness implying finite dimensionality of such potential complex.
	
\noindent
{\bf (D)}	For item (\ref{d:NPC7}) we do not need to assume finite dimensionality. Although systolic complexes
	can be infinitely dimensional, they behave in many ways as (asymptotically) $2$-dimensional objects (see e.g.\ \cite{JS-filling,Osajda-Ideal,OS-AHA,Osajda-simphatic}), and we
	believe this nature makes the Conjecture valid for them. In Theorem~\ref{thm:systolic_elliptic} we provide a partial justification for this belief.  

\noindent
{\bf (E)} There are many more spaces that should be added to the class $\NPC$. The reader is welcomed to add her or his favorite 
NPC-like spaces to the list.

\subsection{Motivations}
\label{s:motiv2}

%

The Meta-Conjecture and Conjecture are related to the following notions and properties, and are generalizations
and unifications of the following (open) questions and conjectures:	

\medskip

\noindent
{\bf{(I)}} A conjecture stating that nonpositively curved groups satisfy the \emph{Tits Alternative}, that is, their finitely generated subgroups are either solvable or contain non-abelian free groups. The statement is quite obvious for Gromov hyperbolic groups. Usually the conjecture is stated for CAT(0) groups; see e.g.\ \cite[Quest~2.8]{Bestvina-prob2},\cite{Bridson-ICM},\cite[Quest~7.1]{Bridson-prob},\cite[Prob~12]{FHT-prob},\cite[Sec.\ 5]{Caprace2014}. The question of ellipticity of locally elliptic actions may be seen as a first step towards the Tits Alternative. For a recent account on the Tits Alternative in the nonpositively curved setting see e.g.\ \cite{OsPrz21,OsPrz21b}. 

\medskip

\noindent
{\bf{(II)}}
A conjecture stating that there are no infinite torsion subgroups of $\mathrm{CAT}(0)$ groups, that is,
groups acting geometrically on $\mathrm{CAT}(0)$ spaces; see e.g.\ \cite{Swenson1999}, \cite[Question 2.11]{Bestvina-prob}, \cite[p.\ 88]{Xie2004}, \cite[Question 8.2]{Bridson-prob}, \cite[Problem 24]{Kapovich-prob}, \cite[\S~IV.5]{Caprace2014}. In \cite[Conjecture 1.5]{NOP-D} the conjecture that every action of a finitely generated torsion group on a finitely dimensional CAT(0) complex is elliptic was formulated for the first time.

\medskip

\noindent
{\bf{(III)}}
A conjectural equivalence of `locally elliptic' and `elliptic' actions, for actions 
on buildings, see e.g.\ \cite{Parreau2003}, \cite[Conjecture 3.22]{Marquis2013}, \cite[Conjecture 1.2]{Marquis2015}. This is related e.g.\ to 
questions concerning bounded subgroups of Kac-Moody groups \cite{Caprace-MAMS}.

\medskip

\noindent
{\bf{(IV)}}
Numerous questions in Algebraic Geometry concerning regularizations or linearizations of certain subgroups of 
the groups of birational transformations of projective surfaces (see e.g.\ \cite[Sec.\ 5.3]{LonUre21}, \cite{Cantat2011}, \cite{Favre2010})) or subgroups of automorphism groups of Aut$({k}^n)$ (see e.g.\ \cite[Introduction]{LamPrz}).

\medskip

\noindent
{\bf{(V)}}
The question of automatic continuity for groups; see e.g.\ \cite{keppeler2021automatic}. Except the group theory the question has origins e.g.\ in the theory of Polish spaces.

\medskip

\noindent
{\bf{(VI)}}
Property (FA) of triviality of group actions on trees introduced by Serre in \cite{Serre73,Serre77,Serre_Trees} and generalized to $\mathbb{R}$-trees in \cite{MorganShalen1984}. ($\mathbb{R}$-)trees are $1$-dimensional
nonpositively curved (more precisely, CAT(0)) spaces. A generalization to higher dimensions is Farb's property FA$_n$ \cite{Farb2009}.

\medskip

\noindent
{\bf{(VII)}}
Kazhdan's property (T), which (for discrete groups) is equivalent to the property of acting trivially on the
Hilbert space, being itself $CAT(0)$. Notably, Kazhdan's property was introduced in order to prove finite generation of some groups.

\medskip

\noindent
{\bf{(VIII)}}
Kaplansky's question on unitarizability of linear groups consisting of conjugates of unitary matrices and, strongly related, Auerbach's problem on boundedness of cyclically bounded groups of matrices,  see e.g.\ \cite{Merz1966,Bass1980,Waterman1988,breuillard_fujiwara}, and some details in Subsection~\ref{s:relres} below.

\medskip

\noindent
{\bf{(IX)}}
General questions concerning the structure and properties of infinite torsion groups. In particular, (not only) we believe (cf.\ e.g.\ \cite{Niblo_1998}) this is related to the following fundamental question about torsion groups: Are there finitely presented infinite torsion groups?

\medskip

\noindent
{\bf{(X)}}
Question of linarity of groups, that is, of existence of faithful finite dimensional representations.
By Schur's theorem \cite{Schu1911}, finitely generated linear torsion groups are finite. (It should be noticed that Tits Alternative takes its name from a result by Tits for linear groups \cite{Tits1972}.)

\subsection{Discussion of the results and state of the art}
\label{s:relres}

The property of trivial actions on nonpositively curved spaces is related to the (FA) property and Kazdan's property (T) -- both, trees and Hilbert spaces are CAT(0).
In particular, by Serre's result \cite{Serre_Trees} groups acting only trivially on NPC spaces are finitely generated. 
Finitely generated torsion groups have property (FA) by \cite{Serre_Trees}. 
However, such groups, even of bounded exponent (that is, Burnside groups) can act without fixed points on infinite dimensional \catz cube complexes (hence, on Hilbert spaces) by \cite{OsajdaDuke}. 
Groups with Kazhdan's property (T) have Property (FA) by \cite{Watatani1982}. However, there are uniform lattices in isometry groups of affine buildings with Kazhdan's property (T).
\mk

Groups (not necessarily finitely generated) acting locally elliptically by isometries on the Euclidean spaces $\mathbb{E}^2,\mathbb{E}^3$ or hyperbolic spaces $\mathbb{H}^2,\mathbb{H}^3$ fix points, cf.\ \cite{Waterman1988,breuillard_fujiwara}. Waterman \cite{Waterman1988} showed the same for $\mathbb{H}^4$. 
He also proved that, for every $n$, a torsion group acting by isometries on $\mathbb{H}^{n}$ fixes a point, thus establishing a part of the Conjecture for complexes isometric to $\mathbb{H}^{n}$.

However, answering in the negative Auerbach's problem on boundedness of cyclically bounded matrix groups Merzlyakov \cite{Merz1966} constructed a two-generated group of locally elliptic isometries of $\mathbb{E}^4$ acting without a fixed point. Similar examples of isometries of $\mathbb{E}^4$ have been independently discovered by Bass \cite{Bass1980} answering a question of Kaplansky, and by Waterman \cite{Waterman1988}, who also constructed an example of a two-generated group acting locally elliptically on 
$\mathbb{H}^5$ without a fixed point. Breuillard-Fujiwara \cite{breuillard_fujiwara} extended such examples to all Euclidean spaces $\mathbb{E}^{2n}$, for $n=2,3,4,\ldots$.

\mk

We believe that the results from the current paper make a significant progress towards the Conjecture. However, even in the \catz setting the Conjecture is widely open. We believe that the case of (Gromov) hyperbolic graphs, as in Corollary~\ref{cor:fixcor}(\ref{cor:fixcor2}) has been known to the experts, and anyway could be established by more standard (that is, not via Helly techniques) methods. However we did not find such a statement in the literature (and we received signals from some colleagues that it was not known to them before). Button \cite{Button} proved the Conjecture in the case of quasitrees, not necessarily being complexes and without further (besides being a quasitree) assumptions on finite dimensionality. 
Observe that any group can act without fixed points on a hyperbolic combinatorial horoball (see e.g.\ \cite{GroMan}) over its Cayley graph.
Such horoball is locally finite, but not uniformly locally finite.
This construction shows also that the assumption on finite dimensionality in our Theorem~\ref{thm:helly_elliptic} cannot be dropped. 
\mk

Imposing some further restrictions on the groups acting, special cases of the Conjecture were established e.g.\ for general
\catz spaces by Caprace-Monod \cite[Lemma 8.1]{CaMo2009}, and for buildings by Parreau \cite[Corollaire 2]{Parreau2000}, \cite[Corollaire 3]{Parreau2003}, and by Marquis \cite{Marquis2013,Marquis2015}. Leder-Varghese~\cite{leder_varghese} observed that it follows from an earlier work by Sageev~\cite{Sageev1995} that the Conjecture holds in the case of \catz cube complexes. In fact, for \catz cube complexes the Conjecture holds also in the infinite dimensional case when there are no infinite cubes \cite{GenLonUre21}. 
Norin-Osajda-Przytycki~\cite{NOP-D} proved the Conjecture for $2$-dimensional \catz complexes with additional mild assumptions.
In particular they proved the conjecture for all $2$-dimen\-sio\-nal (discrete) buildings, and they proved that locally elliptic actions of finitely generated torsion groups on $2$-dimen\-sio\-nal \catz complexes are elliptic. 
Schillewaert-Struyve-Thomas~\cite{schillewaert2020fixed} extended such results beyond the  discrete case, proving them for all $\widetilde{A}_1\times \widetilde{A}_1$, $\widetilde{A}_2$ and $\widetilde{C}_2$ buildings.
\mk

Our Corollary~\ref{cor:fixcor}(\ref{cor:fixcor3}) proves the Conjecture~(\ref{d:NPC11}) in the case of locally finite graphical \cftf complexes. Duda~\cite{DudaC4T4} obtained similar results for not necessarily locally finite classical \cftf complexes and not necessarily finitely generated groups acting on them. In particular, he showed that (not necessarily finitely generated) torsion subgroups of groups with (classical) \cftf small cancellation presentations are finite. In the `hyperbolic' case of $C'(1/4)$-$T(4)$ complexes the conjecture was proved in \cite[Theorem 8.15]{Genev21}.

Regarding Conjecture~\ref{conj:buildings}, let us mention that for some Coxeter groups the combinatorial dimension was computed in \cite{Engberding}.

%
%

\section{Helly graphs and combinatorial dimension} \label{sec:Helly}

A connected graph $X$ is called \emph{Helly} if any family of pairwise intersecting combinatorial balls of $X$ has a non-empty global intersection. We will consider $X$ as its vertex set, and we will endow $X$ with induced graph metric. We refer the reader to~\cite{Helly} for a presentation of Helly graphs and Helly groups.

	One may think of Helly graphs as a very nice class of nonpositively curved, combinatorially defined spaces. Surprisingly enough, many nonpositive curvature metric spaces and groups have a very close relationship to Helly graphs or their non-discrete counterpart, injective metric spaces.
	
	For instance, the thickening of any CAT(0) cube complex is a Helly graph (see~\cite{bandelt_vandevel_superextensions}, and also~\cite[Corollary~3.6]{hruskawise:packing}). Lang showed that the any Gromov hyperbolic group acts properly cocompactly on the Helly hull of any Cayley graph (see~\cite{Lang2013,Helly}). Huang and Osajda proved that any weak Garside group and any Artin group of type FC has a proper and cocompact action on a Helly graph (see~\cite{huang_osajda_helly}). Osajda and Valiunas proved that any group that is hyperbolic relative to Helly groups is Helly (see~\cite{osajda_valiunas}). Haettel, Hoda and Petyt proved that any hierarchically hyperbolic group, and in particular any mapping class group of a surface, has a proper and cobounded action on an injective metric space, see~\cite{haettel_hoda_petyt}.

\mk

Concerning Euclidean buildings, recall the following statement.

\bthm[Hirai, Chalopin et al, Haettel] \label{thm:buildings_helly}
The thickening of any Euclidean building of type $\tilde{A}$ extended, $\tilde{B}$, $\tilde{C}$ or $\tilde{D}$ is Helly.
\ethm

Hirai, and Chalopin et al. proved the case of Euclidean buildings of type $\tilde{A}$ extended and $\tilde{C}$, see~\cite{hirai_uniform_modular} and \cite{CCHO}. In~\cite{haettel_injective_buildings} and \cite{haettel_helly_kpi1}, Haettel proved the statement for all Euclidean buildings of type $\tilde{A}$ extended, $\tilde{B}$, $\tilde{C}$ or $\tilde{D}$. There is an analogous result for classical symmetric spaces, see~\cite{haettel_injective_buildings} for a precise statement.

\mk

The analogy between hyperbolic groups, CAT(0) groups and Helly groups is quite rich, as the following survey of results show.

\bthm[Analogies hyperbolic / CAT(0) / Helly]
Assume that a finitely generated group $G$ acts properly and cocompactly by isometries on a hyperbolic space, CAT(0) space or a Helly graph. Then
\bit
\item $G$ is semi-hyperbolic in the sense of Alonso-Bridson (see~\cite{alonsobridson:semihyperbolic} and \cite{Lang2013}), which has numerous consequences (see~\cite[III.$\Gamma$.4]{BriHaf1999}).
\item $G$ has finitely many conjugacy classes of finite subgroups (see~\cite{BriHaf1999} and \cite{Lang2013}).
\item Asymptotic cones of $G$ are contractible (see~\cite[Theorem~1.5]{Helly}).
\item $G$ admits an EZ-structure (see~\cite{Bestvina96} and \cite[Theorem~1.5]{Helly}).
\item $G$ satisfies the Farrell-Jones conjecture (see~\cite{KasRup17} and \cite[Theorem~1.5]{Helly}).
\item $G$ satisfies the coarse Baum-Connes conjecture (see~\cite{FukOgu20} and \cite[Theorem~1.5]{Helly}).
\item $G$ is finitely presented and has at most quadratic Dehn function (see~\cite[Proposition~III.$\Gamma$.1.6]{BriHaf1999} and \cite[Theorem~1.5]{Helly}).
\item $G$ has type $F_\infty$ (see~\cite[Proposition~II.5.13]{BriHaf1999} and \cite[Theorem~1.5]{Helly}).
\eit 
\ethm

Recall that a geodesic metric space is called \emph{injective} if any family of pairwise intersecting closed balls has a non-empty global intersection. We refer the reader to~\cite{Lang2013} for a presentation of injective metric space, and also the following result of Isbell.

\bthm[\cite{Isbell64}]
Let $X$ denote a metric space. Then there exists an essentially unique minimal injective space $EX$ containing $X$, called the injective hull of $X$.
\ethm

Given any metric space $X$, the \emph{combinatorial dimension} $\operatorname{combdim}(X)$ of $X$ is the topological dimension of the injective hull $EX$ of $X$ (this has been defined by Dress, see~\cite{Dress1984}). This is a quantity that is surprisingly subtle to handle. There is another explicit definition of the combinatorial dimension, which we recall below.

\bthm \cite[Theorem~4.1]{DesLang2015}
A metric space $X$ has combinatorial dimension at most $n$ if and only if, for set $Z \subset X$ with cardinality $2(n+1)$, and for any fixed point free involution $i:Z \ra Z$, there exists a fixed point free bijection $j:Z \ra Z$ distinct from $i$ such that:
$$\sum_{z \in Z} d(z,i(z)) \leq \sum_{z \in Z} d(z,j(z)).$$
\ethm

One very simple remark is that, if a metric space $X$ is isometrically embedded in $Y$, then $\operatorname{combdim}(X) \leq \operatorname{combdim}(Y)$.

\mk

For a Helly graph $X$, the combinatorial dimension of $X$ is the combinatorial dimension of its vertex set, endowed with the combinatorial distance. We now describe more explicitly the combinatorial dimension of a Helly graph $X$. Say that a clique in $X$ is a \emph{round clique} if it is an intersection of balls of $X$. Notice that round cliques are partially ordered by inclusion.

\bthm \cite[Corollary~4.2]{haettel_automorphisms_helly} \label{ref:thm_combdim_cliques}
Let $X$ denote a Helly graph. Then the combinatorial dimension of $X$ is equal to the length of the longest chain of round cliques of $X$.
\ethm

For instance, any tree (even locally infinite) has combinatorial dimension $1$. More generally, the thickening of any $N$-dimensional CAT(0) cube complex has combinatorial dimension $N$.

\medskip

The proof that Corollary~\ref{cor:fixcor} follows from Theorem~\ref{thm:helly_elliptic} is now quite clear, as many spaces embed naturally in Helly graphs.

\noindent
{\emph{Proof of Corollary~\ref{cor:fixcor}.} The result follows from Theorem~\ref{thm:helly_elliptic}, and the aforementioned constructions of Helly graphs associated with the corresponding complexes for: buildings \cite{CCHO,hirai_uniform_modular,haettel_injective_buildings,haettel_helly_kpi1}, hyperbolic graphs \cite{Lang2013,Helly}, small cancellation complexes \cite{Helly}, Salvetti complexes \cite{huang_osajda_helly} and Deligne complexes \cite{haettel_huang_garside_artin_product_Z}.
\hfill $\square$

%

\section{Subdivisions of Helly graphs and elliptic subgroups} \label{sec:subdivisions}

We gather here results from~\cite{haettel_automorphisms_helly}, which we will be useful for the study of elliptic automorphism groups of Helly graphs. Let us start with simple equivalent characterizations of  elliptic automorphisms.

\bpro \cite[Proposition~5.1]{haettel_automorphisms_helly} \label{pro:characterizations_elliptic}
Let $G$ denote a group of automorphisms of a Helly graph $X$. The following are equivalent:
\ben
\item $G$ stabilizes a clique in $X$,
\item $G$ fixes a point in the injective hull $EX$ of $X$ and
\item $G$ has a bounded orbit in $X$.
\een
Such a group is called an \emph{elliptic} group of automorphisms of $X$.
\epro

In the case where the combinatorial dimension is finite, one can say more. More precisely, there is a nice way to describe a Helly subdivision of the injective hull.

\bthm \cite[Theorem~3.1]{haettel_automorphisms_helly} \label{thm:helly_subdivision}
Let $X$ denote a Helly graph with combinatorial dimension $N$. There exists a Helly graph $X'_N$ of combinatorial dimension $N$, called the $N^\text{th}$ Helly subdivision of $X$, with a homothetic embedding $\iota:X \ra X'_N$, such that the following hold.
\bit
\item The map $\iota$ extends to a homothetic bijection between the injective hulls $E(X)$ and $E(X'_N)$.
\item There exists a group embedding from $\operatorname{Aut}(X)$ into $\operatorname{Aut}(X'_N)$ such that $\iota$ is equivariant.
\item A group $G$ of automorphisms of $X$ is elliptic if and only if $G$ fixes a vertex of $X'_N$.
\eit
\ethm

In fact, as detailed in \cite{haettel_automorphisms_helly}, the vertex set of the first Helly subdivision $X'_1$ coincides with the set of all round cliques of $X$.

\mk

Moreover, we can use these subdivisions to study pairs of elliptic subgroups.

\bpro \cite[Proposition~6.1]{haettel_automorphisms_helly} \label{pro:distance_fixed_points_realized_vertices}
Let $X$ denote a Helly graph with combinatorial dimension $N$. Given two elliptic groups $G,H$ of automorphisms of $X$, there are vertices $x,y \in X'_{2N}$ fixed by $G,H$ respectively, such that $d(x,y)$ is equal to the distance between the fixed point sets in $E(X'_{2N})$ of $G,H$ respectively.
\epro

We can also use the $N^\text{th}$ Helly subdivision to study hyperbolic automorphisms of $X$.

\bpro \label{pro:characterizations_hyperbolic} \cite[Proposition~5.4]{haettel_automorphisms_helly}
Let $g$ denote an automorphism of a Helly graph $X$ with finite combinatorial dimension $N$. The following are equivalent:
\ben[1.]
\item $g$ is hyperbolic, i.e for some vertex $x \in X$, the map $n \in \Z \mapsto g^n \cdot x \in X$ is a quasi-isometric embedding.
\item $g$ has a geodesic axis in the injective hull $EX$ of $X$.
\item There exists a vertex $x$ of the $N^\text{th}$ Helly subdivision of $X'_N$ of $X$ and $L \in \N \bs \{0\}$ such that $\forall n \in \N, d(x,g^n \cdot x)=nL$.
\item $g$ has unbounded orbits in $X$.
\een
\epro

In particular, we see that we have the following dichotomy.

\bcor
Let $X$ denote a Helly graph with finite combinatorial dimension. Then any automorphism of $X$ is either elliptic or hyperbolic.
\ecor

\section{Clique-paths} \label{sec:clique_paths}

We now turn to the description of clique-paths in Helly graphs. They will be a subtle variation of normal clique-paths from~\cite{Helly} and from normal cube paths from~\cite{nibloreeves:groups}. They will be flexible enough to allow any hyperbolic automorphism to admit an invariant clique-path (see Proposition~\ref{pro:hyperbolic_clique_path}), and still allow a local-to-global rigidity (see Theorem~\ref{thm:local_to_global_clique_path}).

\mk

Let $X$ be a graph, and let $\sigma,\tau$ be cliques in $X$. We will say that $\sigma,\tau$ are at \emph{transverse distance} $n \geq 0$, denoted $\sigma \lessgtr \tau = n$, if the following hold:
\bit
\item for every $x \in \sigma$ and $y \in \tau$, we have $d(x,y) \geq n$,
\item there exists $x \in \sigma$ such that $\forall y \in \tau, d(x,y) =n$ and
\item there exists $y \in \tau$ such that $\forall x \in \sigma, d(x,y) =n$.
\eit

Say that a sequence of pairwise disjoint cliques $\sigma_0,\dots,\sigma_n$ is an $L$-\emph{clique-path} if for some $L \geq 1$ the following hold:
\bit
\item For every $0 \leq i < j \leq n$, we have $\sigma_i \lessgtr \sigma_j = (j-i)L$. 
\item For any $1 \leq i \leq n-1$, $\sigma_i$ is a clique satisfying $\sigma_i \lessgtr \sigma_{i+1} = L$ and $\sigma_i \lessgtr \sigma_{i-1} = L$, and $\sigma_i$ is a maximal such clique.
\eit

\begin{figure}
\begin{center}
\begin{tikzpicture}
\def \p {0.05}
\def \pp {0.2}
\def \op {0.1}
\def \gris {black!10}
\def \c {blue}

\draw[fill] (0,0) circle (\p) node(a) {};
\draw[fill] (-1,1) circle (\p) node(b) {};
\draw[fill] (0,3) circle (\p) node(c) {};
\draw[fill] (2,1) circle (\p) node(d) {};
\draw[fill] (2,4) circle (\p) node(e) {};
\draw[fill] (4,3) circle (\p) node(f) {};
\draw[fill] (5,5) circle (\p) node(g) {};
\draw[fill] (7,5) circle (\p) node(h) {};
\draw[fill] (7,4) circle (\p) node(i) {};

\draw[\c,fill opacity=\op,fill=\c] (a.center) -- (b.center) -- (c.center) -- (a.center);
\draw[\c,fill opacity=\op,fill=\c] (d.center) -- (e.center) -- (f.center) -- (d.center);
\draw[\c,fill opacity=\op,fill=\c] (g.center) -- (h.center) -- (i.center) -- (g.center);

\draw (c.center) -- (d.center);
\draw (c.center) -- (e.center);
\draw (c.center) -- (f.center);
\draw (a.center) -- (d.center);
\draw (b.center) -- (d.center);
\draw (d.center) -- (g.center);
\draw (e.center) -- (g.center);
\draw (f.center) -- (g.center);
\draw (f.center) -- (h.center);
\draw (f.center) -- (i.center);

\draw (c.center) circle (\pp);
\draw (d.center) circle (\pp); 
\draw (f.center) circle (\pp); 
\draw (g.center) circle (\pp); 

\node at (0,-0.5) {$\sigma_i$};
\node at (3.7,1.8) {$\sigma_{i+1}$};
\node at (7,3.5) {$\sigma_{i+2}$};

\end{tikzpicture}
\end{center}
\caption{A $1$-clique-path $\sigma_i,\sigma_{i+1},\sigma_{i+2}$, where vertices connecting to neighbouring cliques are circled.}
\label{fig:clique_path}
\end{figure}
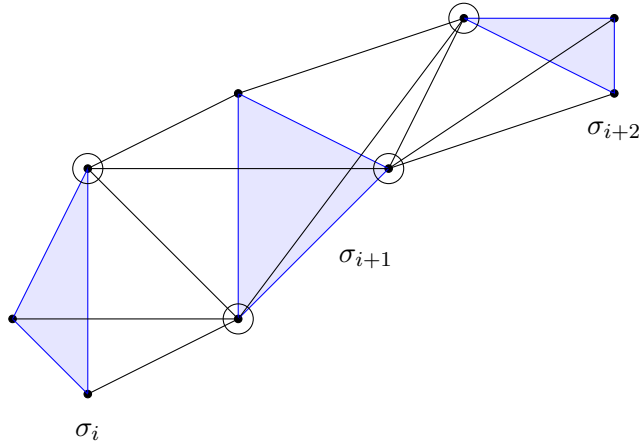

\mk

Say that a sequence of cliques $\sigma_0,\dots,\sigma_n$ is a \emph{local $L$-clique-path} if, for any $0 \leq i \leq n-2$, the sequence $\sigma_i,\sigma_{i+1},\sigma_{i+2}$ is an $L$-clique-path. In other words:
\bit
\item For any $0 \leq i \leq n-1$,  $\sigma_i \lessgtr \sigma_{i+1} = L$.
\item For any $1 \leq i \leq n-1$, $\sigma_i$ is a clique satisfying $\sigma_i \lessgtr \sigma_{i+1} = L$ and $\sigma_i \lessgtr \sigma_{i-1} = L$, and $\sigma_i$ is a maximal such clique.
\item For any $0 \leq i \leq n-2$, $\sigma_i \lessgtr \sigma_{i+2} = 2L$.
\eit

Observe that an $L$-clique-path is a {local $L$-clique-path}. In Theorem~\ref{thm:local_to_global_clique_path} below we will prove a version of a converse statement.

If $L=1$, we will also simply say (local) clique-path. See Figure~\ref{fig:clique_path}.

\mk

Let us first remark that clique-paths involve only round cliques as defined in Section~\ref{sec:Helly}, i.e. intersections of balls.

\blem
Let $X$ denote a graph, $L \geq 1$, and $\sigma_0,\sigma_1,\sigma_2$ an $L$-clique path in $X$. Then $\sigma_1$ is a round clique, i.e. it is an intersection of balls.
\elem

\bp
By definition of an $L$-clique path, there exist $x_0 \in \sigma_0$ and $x_2 \in \sigma_2$ such that $\sigma_1 \subset B(x_0,L) \cap B(x_2,L)$. By maximality of $\sigma_1$, we deduce that
$$\sigma_1 = B(x_0,L) \cap B(x_2,L) \cap \bigcap_{x \in \sigma_1} B(x,1).$$
Hence $\sigma_1$ is a round clique.
\ep

\mk

\blem \label{lem:existence_path}
Fix a Helly graph $X$ with finite combinatorial dimension. Consider any two cliques $\sigma,\tau$ in $X$ such that $\sigma \lessgtr \tau = nL$, for some integers $n,L \in \N \bs \{0\}$. There exists an $L$-clique-path $\sigma_0=\sigma,\sigma_1,\dots,\sigma_n=\tau$ between $\sigma$ and $\tau$ of length $n$. Moreover, if $x_0 \in \sigma$ and $x_n \in \tau$ are such that $\tau \subset B(x_0,nL)$ and $\sigma \subset B(x_n,nL)$, one can choose a $L$-clique-path $\sigma_0=\sigma,\sigma_1,\dots,\sigma_n=\tau$ such that $\sigma_1 \subset B(x_0,L)$ and $\sigma_{n-1} \subset B(x_n,L)$.
\elem

\bp
We will first prove the statement for $L=1$, and by induction on $n \geq 1$. For $n = 1$ there is nothing to prove.

\mk

By induction, assume that $n \geq 2$, and assume that the statement is true for $n-1$. Consider cliques $\sigma,\tau$ such that $\sigma \lessgtr \tau = n$. Let $x_0 \in \sigma_0=\sigma$ be such that $\sigma_n=\tau \subset B(x_0,n)$, and let $x_n \in \sigma_n$ be such that $\sigma_0 \subset B(x_n,n)$.

The balls $B(t_0,n-1)$, for $t_0 \in \sigma_0$, and $B(x_n,1)$ pairwise intersect, so as $X$ is a Helly graph there exists $x \in \cap_{t_0 \in \sigma_0} B(t_0,n-1) \cap B(x_n,1)$. Now the balls $B(t_n,1)$, for $t_n \in \sigma_n$, $B(x_0,n-1)$ and $B(x,1)$ pairwise intersect, so there exists $y \in \cap_{t_n \in \sigma_n} B(t_n,1) \cap B(x_0,n-1) \cap B(x,1)$. So the clique $\sigma'_{n-1}=\{x,y\}$ satisfies $\sigma'_{n-1} \lessgtr \sigma_0=n-1$ and $\sigma_{n-1} \lessgtr \sigma_n=1$.

\mk

By induction, there exists a clique-path $\sigma'_0,\sigma'_1,\dots,\sigma'_{n-1}$ from $\sigma'_0=\sigma_0$ to $\sigma'_{n-1}$ of length $n-1$. 

\mk

Let $\sigma'_n=\sigma_n$. For any $0 \leq i \leq n-1$,  we have $\sigma'_i \lessgtr \sigma'_{i+1} = 1$. Therefore, since $X$ has finite combinatorial dimension, according to Theorem~\ref{ref:thm_combdim_cliques}, there is a bound on the length of chains of round cliques. So we may consider a maximal family of cliques $\sigma_1,\dots,\sigma_{n-1}$ containing $\sigma'_1,\dots,\sigma'_{n-1}$ respectively, such that for any $1 \leq i \leq n-1$, we have $\sigma_i \lessgtr \sigma_{i-1} = 1$ and $\sigma_i \lessgtr \sigma_{i+1}=1$.

\mk

We will prove that $\sigma_0,\sigma_1,\dots,\sigma_n$ is a $1$-clique-path. For any $0 \leq i < j \leq n$, let $y_i \in \sigma_i$ and $y_j \in \sigma_j$. Since $d(x_0,y_i) + d(y_i,y_j)+d(y_j,x_n) \geq d(x_0,x_n)=n$, we deduce that $d(y_i,y_j) \geq j-i$. So we deduce that $\sigma_i \lessgtr \sigma_j=j-i$.

\mk

Hence $\sigma_0,\sigma_1,\dots,\sigma_{n-1},\sigma_n$ is a $1$-clique-path from $\sigma$ to $\tau$.

\mk

Now consider the case $L \geq 2$, and consider cliques $\sigma,\tau$ such that $\sigma \lessgtr \tau = nL$. According to the case $L=1$, there exists a $1$-clique path $\sigma'_0=\sigma,\sigma'_1,\dots,\sigma'_{nL} = \tau$ from $\sigma$ to $\tau$ of length $nL$. For every $0 \leq i \leq n-1$, we have $\sigma'_{iL} \lessgtr \sigma'_{i+1} = L$. Since there is a bound on the length of chains of round cliques, we may consider a maximal family of cliques $\sigma_1,\dots,\sigma_{n-1}$ containing $\sigma'_{L},\dots,\sigma'_{(n-1)L}$ respectively, such that for any $1 \leq i \leq n-1$, we have $\sigma_i \lessgtr \sigma_{(i-1)L} = L$ and $\sigma_i \lessgtr \sigma_{i+1}=L$. We then argue as before that, for any $0 \leq i < j \leq n$, we have $\sigma_i \lessgtr \sigma_j = (j-i)L$. So $\sigma_0=\sigma,\sigma_1,\dots,\sigma_n=\tau$ is a $L$-clique path from $\sigma$ to $\tau$ of length $n$.
\ep

We will now prove the main local-to-global property of clique-paths, starting with the case of local $1$-clique-paths.

\bpro \label{pro:local_clique_path_between_vertices_L1}
Fix a Helly graph $X$, $n \geq 2$, and consider a local clique-path $\sigma_0,\dots,\sigma_n$ in $X$. Let $x_0 \in \sigma_0$ such that $\sigma_1 \subset B(x_0,1)$, and let $x_n \in \sigma_n$ such that $\sigma_{n-1} \subset B(x_n,1)$. Then $\{x_0\},\sigma_1,\sigma_2,\dots,\sigma_{n-1},\{x_n\}$ is a clique-path.
\epro

\bp
We will prove the statement by induction on $n \geq 2$. For $n=2$ any local clique-path is a clique-path, so fix $n \geq 3$ and assume that the statement holds for $n-1$. Fix a local clique-path $\sigma_0,\dots,\sigma_n$.

\mk

By induction, we only need to prove that $\{x_0\} \lessgtr \sigma_{n-1}=n-1$, $\{x_n\} \lessgtr \sigma_{0}=n-1$ and $d(x_0,x_n)=n$. The first two equalities are consequences of $d(x_0,x_n)=n$: we will prove it.

\mk

By contradiction, assume that $d(x_0,x_n) \leq n-1$. The balls $B(x_0,n-2)$, $B(x_n,1)$, $B(y_{n-2},1)$ for $y_{n-2} \in \sigma_{n-2}$ and $B(y_{n-1},1)$ for $y_{n-1} \in \sigma_{n-1}$ pairwise intersect: since $X$ is a Helly graph, we may consider
$$ z \in B(x_0,n-2) \cap B(x_n,1) \cap \bigcap_{y_{n-2} \in \sigma_{n-2}} B(y_{n-2},1) \cap \bigcap_{y_{n-1} \in \sigma_{n-1}} B(y_{n-1},1).$$
We will prove that $z \in \sigma_{n-1}$.

\mk

The vertex $z$ is adjacent to $\sigma_{n-1}$, let $\sigma'_{n-1}=\sigma_{n-1} \cup z$. Since $\sigma_{n-2} \lessgtr \sigma_n = 2$, we know that $\sigma'_{n-1}$ is disjoint from $\sigma_{n-2}$ and $\sigma_n$.

We have $\sigma'_{n-1} \lessgtr \sigma_{n}=1$ and $\sigma'_{n-1} \lessgtr \sigma_{n-2}=1$. So by maximality of $\sigma_{n-1}$, we deduce that $\sigma'_{n-1}=\sigma_{n-1}$, hence $z \in \sigma_{n-1}$.

\mk

Since $\sigma_{n-2} \subset B(z,1)$, by induction we know that $\{x_0\},\sigma_1,\dots,\sigma_{n-2},\{z\}$ is a clique-path. In particular $d(x_0,z)=n-1$, which contradicts $d(z,x_0) \leq n-2$.

\mk

So we conclude that $d(x_0,x_n)=n$. This implies that $\{x_0\},\sigma_1,\sigma_2,\dots,\sigma_{n-1},\{x_n\}$ is a clique-path.
\ep

We now prove a very similar result for local $L$-clique-paths.

\bpro \label{pro:local_clique_path_between_vertices_Lqcq}
Fix a Helly graph $X$ with finite combinatorial dimension, $L \geq 1$ and $n \geq 2$. Consider a local $L$-clique-path $\sigma_0,\dots,\sigma_n$ in $X$. There exist $x_0 \in \sigma_0$ and $x_n \in \sigma_n$ such that $\sigma_1 \subset B(x_0,L)$, $\sigma_{n-1} \subset B(x_n,L)$ and $\{x_0\},\sigma_1,\sigma_2,\dots,\sigma_{n-1},\{x_n\}$ is a $L$-clique-path.
\epro

\bp
According to Lemma~\ref{lem:existence_path}, for each $0 \leq i \leq n-1$, there exists a $1$-clique-path between $\sigma_i$ and $\sigma_{i+1}$. So there exists a sequence of cliques $\tau_0,\tau_1,\dots,\tau_{nL}$ such that:
\bit
\item For each $0 \leq i \leq n$, we have $\tau_{iL}=\sigma_i$.
\item For each $0 \leq i \leq n-1$, the sequence $\tau_{iL},\tau_{iL+1},\dots,\tau_{(i+1)L}$ is a $1$-clique-path.  
\eit

We will prove that $\tau_0,\tau_1,\dots,\tau_{nL}$ is a local $1$-clique-path.

\mk

We first prove that, for any $1 \leq i \leq n-1$, we have $\tau_{iL-1} \lessgtr \tau_{iL+1} = 2$. Fix $y_{iL-1} \in \tau_{iL-1}$ and $y_{iL+1} \in \tau_{iL+1}$. Let $x_{i-1} \in \sigma_{i-1}$ such that $\tau_{iL-1} \subset B(x_{i-1},L-1)$ and $x_{i+1} \in \sigma_{i+1}$ such that $\tau_{iL+1} \subset B(x_{i+1},L-1)$. Then $d(x_{i-1},y_{iL-1})+d(y_{iL-1},y_{iL+1})+d(y_{iL+1},x_{i+1}) \geq d(x_{i-1},x_{i+1})=2L$, so $d(y_{iL-1},y_{iL+1}) \geq 2$. Hence  $\tau_{iL-1} \lessgtr \tau_{iL+1} = 2$.

\mk

We now prove that, for any $1 \leq i \leq n-1$, $\tau_{iL}$ is a maximal clique such that $\tau_{iL} \lessgtr \tau_{iL-1}=1$ and $\tau_{iL} \lessgtr \tau_{iL+1}=1$. Assume that $\tau'_{iL}$ is a clique containing $\tau_{iL}$ such that $\tau'_{iL} \lessgtr \tau_{iL-1}=1$ and $\tau'_{iL} \lessgtr \tau_{iL+1}=1$. Since $\sigma_{i-1} \lessgtr \sigma_{i+1}=2L$, we deduce that $\sigma_{i-1} \lessgtr \tau'_{iL}=L$ and $\sigma_{i+1} \lessgtr \tau'_{iL}=L$. Since $\sigma_0,\sigma_1,\dots,\sigma_n$ is a local $L$-clique-path, by maximality of $\sigma_i$ we deduce that $\tau'_{iL}=\sigma_i=\tau_{iL}$.

\mk

So we have proved that $\tau_0,\tau_1,\dots,\tau_{nL}$ is a local $1$-clique-path. Let $x_0 \in \tau_0=\sigma$ such that $\tau_1 \subset B(x_0,1)$, and let $x_n \in \tau_{nL}=\tau$ such that $\tau_{nL-1} \subset B(x_n,1)$. According to Proposition~\ref{pro:local_clique_path_between_vertices_L1}, $\{x_0\},\tau_1,\dots,\tau_{n-1},\{x_n\}$ is a $L$-clique-path.

\mk

Then for any $1 \leq i < j \leq n-1$, we have $\tau_{iL} \lessgtr \tau_{jL}=(j-i)L$, so $\sigma_i \lessgtr \sigma_j = (j-i)L$. We deduce that $\{x_0\},\sigma_1,\sigma_2,\dots,\sigma_{n-1},\{x_n\}$ is a $L$-clique-path.\ep

So we immediately deduce from Proposition~\ref{pro:local_clique_path_between_vertices_Lqcq} the following crucial local-to-global result.

\bthm \label{thm:local_to_global_clique_path}
In a Helly graph with finite combinatorial dimension, for any $L \geq 1$, any bi-infinite local $L$-clique-path is an $L$-clique-path.
\ethm

We also deduce that any hyperbolic automorphism of a Helly graph has an invariant clique-path.

\bpro \label{pro:hyperbolic_clique_path}
Fix a Helly graph $X$ with finite combinatorial dimension $N$, and let $g$ denote a hyperbolic automorphism of $X$. There exist integers $a,L \in \N \bs \{0\}$ and a bi-infinite $L$-clique-path $(\sigma_n)_{n \in \Z}$ in the $N^\text{th}$ Helly subdivision of $X'_N$ of $X$ such that $\forall n \in \Z, g^a \cdot \sigma_n=\sigma_{n+1}$.
\epro

\bp
According to Proposition~\ref{pro:characterizations_hyperbolic}, up to replacing $X$ with its $N^\text{th}$ Helly subdivision of $X'_N$, there exist integers $a,L \in \N \bs \{0\}$ and a vertex $x \in X$ such that $\forall n \in \N, d(x,g^{an} \cdot x)=nL$.

\mk

According to Theorem~\ref{ref:thm_combdim_cliques}, since $X$ has finite combinatorial dimension, there is a bound on the length of chains of round cliques. Let us then consider a clique $\sigma$ in an $L$-clique path from $x$ to $g^{an_0} \cdot x$, which is maximal with respect to the inclusion, over all choices of $n_0 \geq 1$ and over all choices of $L$-clique paths.

Since $\sigma \lessgtr g^{an_0} \cdot \sigma = n_0L$, Let us consider an $L$-clique path from $\sigma_0=\sigma$ to $\sigma_{n_0}=g^{an_0} \cdot \sigma$. Now consider the binfinite family of cliques $(\sigma_n)_{n \in \Z}$ of $X$ defined by $\sigma_{n+n_0}=g^{an_0} \cdot \sigma_n$ for all $n \in \Z$. It is a local $L$-clique-path, except possibly at $\sigma_{kn_0}$, for $k \in \Z$. Since $\sigma$ is maximal with respect to inclusion, we deduce that it is also a local $L$-clique-path at $\sigma_0=\sigma$. Hence $(\sigma_n)_{n \in \Z}$ is a local $L$-clique-path.

\mk

According to Theorem~\ref{thm:local_to_global_clique_path}, $(\sigma_n)_{n \in \Z}$ is a $L$-clique path, which is invariant by $g^a$.
\ep

\section{Linear orbit growth} \label{sec:linear_orbit_growth}

We will now prove a simple result about elliptic groups of isometries of an injective metric space ensuring that orbits grow linearly.

\blem \label{lem:diameter_doubles}
Let $X$ be an injective metric space, and let $G,H$ denote elliptic isometry groups of $X$, i.e.\ with fixed points in $X$. For any $n \in \N$ and any $x \in X^G$, there exists $y \in X^G$ such that
\beq \diam((GH)^{2^{n-1}} \cdot y) & \leq & \f{1}{2}\diam((GH)^{2^{n}} \cdot x) \mbox{ if $n \geq 1$} \\
\diam(H \cdot y) & \leq & \f{1}{2}\diam((GH)^{2^{n}} \cdot x) \mbox{ if $n =0$},\eeq
where $(GH)^k$ denotes $(GH)^k=\{g_1 h_1 g_2 h_2 \cdots g_k h_k \st g_1,\dots,g_k \in G, h_1,\dots,h_k \in H\}$.
\elem

\bp
Assume that $2L = \diam((GH)^{2^{n}} \cdot x)$ is finite, otherwise the conclusion holds trivially. Then, for any $u,v \in (GH)^{2^{n}}$, we have $d(u\cdot x, v \cdot x) \leq 2L$. In particular, since $X$ is injective, the intersection $\bigcap_{u \in (GH)^{2^{n}}} B(u \cdot x,L)$ is non-empty. Since this intersection is bounded and $G$-invariant, according to~\cite[Proposition~1.2]{Lang2013}, there exists $z \in \bigcap_{u \in (GH)^{2^{n}}} B(u \cdot x,L)$ that is fixed by $G$.

So for every $u \in (GH)^{2^{n}}$, we have $d(z,u \cdot x) \leq L$. Hence we deduce that for every $u,v \in K_n$ we have $d(u \cdot z,v \cdot x) \leq L$, where $K_n=(GH)^{2^{n-1}}$ if $n \geq 1$ and $K_n=H$ if $n=0$.

According to~\cite[Proposition~3.8]{Lang2013}, there exists a conical geodesic reversible bicombing $\gamma$ on $X$, equivariant with respect to isometries of $X$. Let $y=\gamma(x,z,\f{1}{2})$ denote the midpoint of $x$ and $z$ with respect to $\gamma$.

For any $u,v \in K_n$, we have $u \cdot y = \gamma(u \cdot x,u \cdot z,\f{1}{2})$ and $v \cdot y = \gamma(v \cdot z,v \cdot x,\f{1}{2})$. By conicality, we deduce that
$$d(u \cdot y,v \cdot y) \leq \f{1}{2}d(u \cdot x,v \cdot z) + \f{1}{2}d(u \cdot z,v \cdot x) \leq L.$$


So the diameter of $K_n \cdot y$ is at most $L$. And by equivariance of $\gamma$, $y$ is fixed by $G$.
\ep

\bpro \label{pro:linear_growth}
Let $X$ be an injective metric space, and let $G,H$ denote elliptic isometry groups of $X$ such that $d(X^G,X^H)=L>0$. Assume that there exists $x \in X^G$ such that $d(x,X^H)=L$. Then, for any $n \in \N$, there exist $g_1,\dots,g_{n} \in G$ and $h_1,\dots,h_{n}\in H$ such that
$$x_0=x,x_1=g_1h_1 \cdot x, x_2=g_1h_1g_2h_2 \cdot x, \dots, x_n=g_1h_1 \cdots g_{n}h_{n} \cdot x$$
lie on a geodesic. More precisely, for any $0 \leq i \leq j \leq n$, we have
$$d(x_i,x_j) = (j-i)2L.$$
\epro

\bp
Note that it is sufficient to prove the statement when $n$ is a power of $2$.

\mk

Let $2L'$ denote the diameter of $H \cdot x$, we will prove that $L'=L$. For any $h,h' \in H$, we have $d(h \cdot x,h' \cdot x) \leq 2L'$, so since $X$ is injective the intersection $\bigcap_{h \in H} B(h \cdot x,L')$ is non-empty. Since this intersection is furthermore $H$-invariant, according to~\cite[Proposition~1.2]{Lang2013}, there exists $y \in \bigcap_{h \in H} B(h \cdot x,L')$ that is fixed by $H$. Since $d(x,y) \leq L'$, we deduce that $L \leq L'$. And if $z \in X^H$ is such that $d(x,z)=L$, for any $h \in H$ we have $d(x,h \cdot x) \leq d(x,z)+d(z,h \cdot x) \leq 2L$, so $\diam(H \cdot x) \leq 2L$. In conclusion $L'=L$, and the diameter of $H \cdot x$ equals $2L$.

\mk

For any $y \in X^G$, as in the previous paragraph, we may argue that the diameter of $H \cdot y$ is at least $2L$. According to Lemma~\ref{lem:diameter_doubles}, by induction on $m \geq 0$, we deduce that the diameter of $(GH)^{2^{m}} \cdot x$ is at least $2^{m+2}L$.

\mk

In particular, for any $m \geq 0$, there exist $g_1,\dots,g_{2^{m+1}} \in G$ and $h_1,\dots,h_{2^{m+1}}\in H$ such that $d(x,g_1h_1 \cdots g_{2^{m+1}}h_{2^{m+1}} \cdot x) \geq 2^{m+2}L$. Since for any $1 \leq i \leq 2^{m+1}$ we have $d(x,g_ih_i \cdot x) \leq 2L$, we deduce that $d(x,g_1h_1 \cdots g_{2^{m+1}}h_{2^{m+1}} \cdot x) = 2^{m+2}L$, and furthermore that for any $0 \leq i \leq j \leq 2^{m+1}$, we have
$$d(g_1h_1 \cdots g_ih_i \cdot x,g_1h_1 \cdots g_jh_j \cdot x) = (j-i)2L.$$
\ep

\section{Locally elliptic actions} \label{sec:locally_elliptic}

We now turn to one of the main results of this article, namely that locally elliptic actions on Helly graphs are globally elliptic. Theorem~\ref{thm:helly_elliptic} is an immediate consequence of the following result, since any graph isometrically embeds in its Helly hull.

\bthm \label{thm:locally_elliptic}
Let $X$ be a Helly graph with finite combinatorial dimension, and assume that $G,H$ are elliptic automorphism groups of $X$. Then either $\<G,H\>$ is elliptic, or there exists a hyperbolic element in $\<G,H\>$.
\ethm

\bp
By assumption, $G$ and $H$ have fixed points in $EX$. Let $D \geq 1$ denote the combinatorial dimension of $X$. According to Proposition~\ref{pro:distance_fixed_points_realized_vertices}, up to replacing $X$ with its Helly subdivision $X'_{2D}$, we may assume that the minimal distance between $(EX)^G$ and $(EX)^H$ is realized by vertices of $X$. Let $x \in X^G$ and $y \in X^H$ denote vertices such that $d(x,y)=L=d(EX^G,EX^H)$.

\mk

If $L=0$, then $x=y$ is fixed by $G$ and $H$, so $\<G,H\>$ is elliptic.

\mk

According to Theorem~\ref{ref:thm_combdim_cliques}, since $X$ has finite combinatorial dimension, there is a bound on the length of chains of round cliques.

Choose $n \geq 1$ and $u \in (GH)^n$ such that $d(x,u \cdot x)=2n L$ and there exists a clique-path from $x$ to $u \cdot x$, containing a clique $\sigma$ which is maximal with respect to inclusion, among all choices of $n \geq 1$, $u=u_1 \dots u_n \in (GH)^n$ and all clique-paths containing the vertices
$$x,u_1 \cdot y, u_1 \cdot x, u_1u_2 \cdot y, u_1u_2 \cdot x, \dots,u_1u_2 \dots u_n \cdot y,u_1u_2 \dots u_n \cdot x.$$
Let $1 \leq k \leq 2nL-1$ such that $\sigma \lessgtr x=k$ and $\sigma \lessgtr u \cdot x=2nL-k$.

\mk

Up to translating by an element of $\<G,H\>$, we may assume that $\sigma$ lies between $x$ and $y$, for some $x \in X^G$ and $y \in X^H$ at distance $L$. According to the proof of Proposition~\ref{pro:linear_growth}, the diameter of $HG \cdot y$ equals $4L$. Hence there exist $g \in G$ and $h \in H$ such that $d(y,gh \cdot x)=3L$. So we have $\sigma \lessgtr hg \cdot \sigma = 2L$. Let us consider a clique-path $\sigma_0,\sigma_1,\dots,\sigma_{2L}$ from $\sigma_0=\sigma$ to $\sigma_{2L}=gh \cdot \sigma$. Now extend this sequence of cliques using the $gh$-action, by letting $\sigma_{n+2L}=gh \cdot \sigma_n$ for all $n \in \Z$.

\mk

This sequence of cliques $(\sigma_n)_{n \in \Z}$ is invariant under translation by $gh$. Moreover, it is a local clique-path except possibly at each clique $\sigma_{2kL}$, where $k \in \Z$. By maximality of the round clique $\sigma$, we deduce that it is also a local clique-path at $\sigma_0=\sigma$. Hence $(\sigma_n)_{n \in \Z}$ is a local clique-path. According to Theorem~\ref{thm:local_to_global_clique_path}, we conclude that it is a clique-path, invariant under $gh \in \<G,H\>$. 

So according to Proposition~\ref{pro:characterizations_hyperbolic}, $gh$ is hyperbolic and it has infinite order.
\ep

	\section{Systolic complexes and graphical small cancellation complexes}
	\label{s:sc}
	Recall that, for $k\geq 6$, a flag simplicial complex $X$ is \emph{$k$-large} if every cycle in (the $1$-skeleton of) $X$ of length smaller than $k$ has a \emph{diagonal}, that is, there exists an edge in $X$ connecting two non-consecutive vertices of the cycle. 
A simply connected $k$-large complex is \emph{$k$-systolic}, and $6$-systolic complexes are usually called simply \emph{systolic}. Observe that $k$-systolicity implies $l$-systolicity, for $k\geq l$.
Such complexes were first considered by Chepoi \cite{Chepoi2000} in relation with \emph{bridged graphs}, which are their $1$-skeleta. The study of groups acting geometrically on systolic complexes was initiated by 
Januszkiewicz-\'Swi\polhk{a}tkowski \cite{JanuszkiewiczSwiatkowski2006} and by Haglund \cite{Haglund2003}. 

In this section we consider the following class of group actions on systolic complexes. We say that
$G$ acts \emph{strongly rigidly} on $X$ if the only $g\in G$ stabilizing a (possibly infinite) simplex of dimension
at least $1$ is the indentity. Such actions appear naturally as actions induced by actions of small cancellation groups on their Cayley graphs as explained below. 

In what follows, unless stated otherwise,  we consider a group $G$ acting strongly rigidly on a $18$-systolic complex $X$. 
\begin{lemma}
	\label{l:fix1}
	Let $g\in G\setminus \{1\}$ fix a vertex in $X$. Then $\fix (g)$ is a vertex in $X$.
\end{lemma}
\begin{proof}
	Suppose there are two distinct vertices $v,w\in X$ in $\fix (g)$. Consider the interval $I:=I(v,w)$ between them. Let $\sigma = I \cap X_v$, where $X_v$ denotes the link of $v$ in $X$. It is a non-empty (possibly infinite) simplex, see e.g.\ \cite[Lemma~2.3]{Osajda-connect}. We have $gI=I$, hence $g\sigma=\sigma$.
	By the strong rigidity of the action we have $g=1$, contradiction.  
\end{proof}

\begin{lemma}
	\label{l:fk4}
	Let  $w$ be a vertex in the link $X_v$, for $v=\fix (g)$. Then there exists a natural number $k>0$ such that $d_{X_v}(w,g^kw)\geq 5$.
\end{lemma}
\begin{proof}
	Suppose not, that is, suppose that for every $k$ we have $d_{X_v}(w,g^kw) \leq 4$. 
	Consider the intersection $Z:=\bigcap_k B_4(g^kw,X_v)$ of balls of radius $4$ around $g^kw)$ in the link $X_v$. It is nonempty and it is contained in $B_8(w,X_v)$.
	By \cite[Lemma 2.5]{Osajda-connect} the ball $B_8(w,X_v)$ is itself a systolic complex and the balls in $X_v$
	of the form $B_4(g^kw,X_v)$ are balls in $B_8(w,X_v)$, hence are convex subcomplexes of $B_8(w,X_v)$.
	Therefore $Z$ is a convex subcomplex of a systolic complex $B_8(w,X_v)$, hence it is systolic itself. Furthermore, $Z$ is $\langle g \rangle$-invariant and
	bounded. It follows that there exists a (possibly infinite) simplex $\sigma$ in $Z$ fixed by $g$. It follows 
	that the simplex of $X$ spanned by $\sigma$ and $v$ is fixed by $g$, contradicting the strict rigidity of the action.
\end{proof}

\begin{lemma}
	\label{l:fix3}
	Let $g,g' \in G\setminus \{1\}$ such that $\fix(g)\neq \fix(g')$. Then there exist natural numbers $k,k'>0$ such that $g^kg'^{k'}$ is an element in $G$ acting with an unbounded orbit on $X$.
\end{lemma}
\begin{proof}
	Choose a $1$-skeleton geodesic $\gamma$ between $v:=\fix(g)$ and $v':=\fix(g')$. Let $w$ and $w'$ denote vertices on $\gamma$ adjacent to, respectively, $v$ and $v'$. 
	By Lemma~\ref{l:fk4} there exist $k,k'$ such that $d_{X_v}(w,g^kw)\geq 5$ and $d_{X_{v'}}(w',{g'}^{k'}w')\geq 5$.
	We claim that $f:=g^kg'^{k'}$ is not elliptic.
	
	Suppose not, that is, suppose that the orbit of   $f$ is bounded.
	Consider the path (see Figure~\ref{fig:systolic})
	\begin{align}
		\alpha:= \bigcup_{i\in\mathbb{Z}}f^i(\gamma \cup g^k\gamma).
	\end{align}
	Since the orbits of $f$ are bounded, there exist vertices $u,u'$ in $\alpha$ whose distance in $\alpha$ is strictly bigger than their distance in $X$. We can choose $u,u'$ with the minimal distance among such pairs.  Let $\alpha'$ be the path being part of $\alpha$ between $u$ and $u'$, and let $\beta$ be a geodesic in $X$ between 
	$u'$ and $u$. Consider the closed path $\alpha'\beta$.
	\begin{figure}[h!]
		\begin{center}
			\includegraphics[scale=0.6]{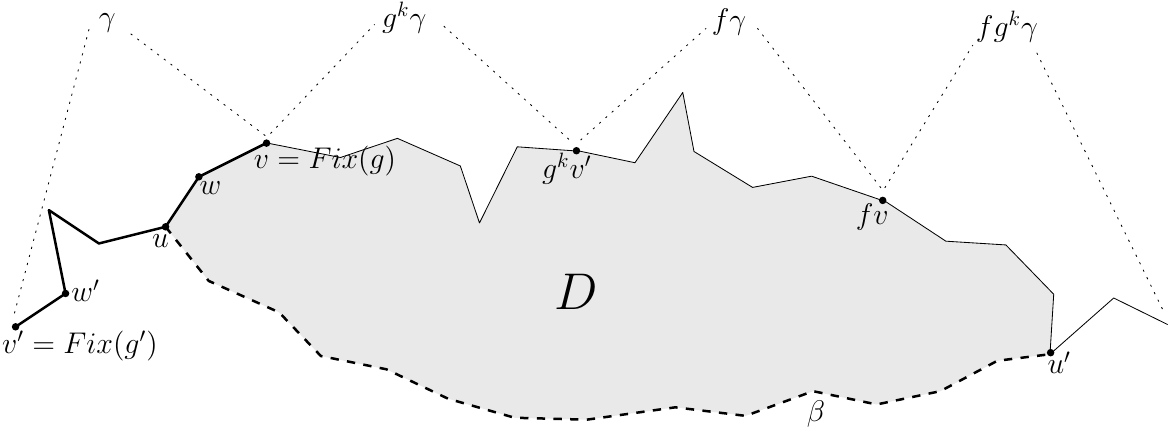}
		\end{center}
		\caption{The path $\alpha$, a shortcut $\beta$ (dashed), and a disk diagram $D$ (shaded).}\label{fig:systolic}
	\end{figure}
	Consider a minimal (singular) disk diagram $D$ for $\alpha'\beta$ (see Figure~\ref{fig:systolic}). Such a diagram is systolic. 
	Recall, that the \emph{defect} of a vertex $z$ on the boundary $\alpha'\beta$ of $D$ is equal to three minus the number of triangles of $D$ containing $z$. By the combinatorial Gauss-Bonnet Theorem and since $D$ is systolic
	we have that the sum of defects of the boundary vertices of $D$ is at least $6$. By the choice of $u,u'$ the sum of their defects is at most $4$. By the choice of $k,k'$ the defects in the points $f^iv,f^iv'$ are at most $-1$.
		The sum of defects along interior vertices of a geodesic (in particular,  $f^i\gamma$ or $f^ig^k\gamma$ or $\beta$) is at most $1$. Since there is a shortcut between $u$ and $u'$,
		the sum of defects of vertices of $\alpha'$ different than $u,u'$ is at most $0$. Therefore, the sum of defects along the boundary vertices is at most $5$, a contradiction.  
\end{proof}

\medskip

\noindent
\emph{Proof of Theorem~\ref{thm:systolic_elliptic}.}
If all the generators of $G$ have a common fixed point then the $G$-action is elliptic. If not, then by Lemma~\ref{l:fix3} there exists a non-elliptic element leading to a contradiction. 
\hfill $\square$

\medskip

In the rest of the section we turn towards actions on graphical small cancellation complexes.
We do not give all the definitions here directing the reader to \cite{OsaPry2018}, whose notation we use in what follows.

Let $Y$ be a simply connected graphical $C(18)$-small cancellation complex. Let $G$ be a finitely generated group acting on $Y$
by automorphisms such that 
the action induces a free action on the $1$-skeleton $Y^{(1)}$ of $Y$. 
Recall \cite{OsaPry2018} that the \emph{Wise complex} $W(Y)$ of $Y$ is the nerve of the covering of $Y$ by (closed) $2$-cells.
We assume here every edge in $Y$ is contained in some $2$-cell -- this is not restrictive since one can always add equivariantly `artificial' $2$-cells preserving small cancellation conditions.
Since $Y$ is  $C(18)$ its Wise complex $W(Y)$ is $18$-systolic by
\cite[Theorem 7.10]{OsaPry2018}.


\begin{lemma}
	\label{l:2rigid}
	The induced $G$-action on the Wise complex $W(Y)$ is strongly rigid. 
\end{lemma}
\begin{proof}
		Let $g\in G$. suppose that there exists a simplex $\sigma=\{ z_0,z_1,\ldots\}$ of dimension at least $1$ in $W(Y)$ with  $g\sigma=\sigma$.
	It follows that for the intersection $Z:=\bigcap z_i \subseteq Y$ we have $gZ=Z$. Such intersection is a non-empty bounded tree,
	hence there exists a vertex or an edge fixed by $g$. By the assumption on freeness of the action
	we have that $g=1$.
\end{proof}

\medskip

\noindent
\emph{Proof of Theorem~\ref{thm:C(13)_elliptic}.}
The action of a $C(18)$ group on the graphical small cancellation complex derived from the Cayley graph
is free on the $1$-skeleton. Hence the theorem follows from Lemma~\ref{l:2rigid} and Theorem~\ref{thm:systolic_elliptic}.	
\hfill $\square$

	\bibliography{mybib}{}
	\bibliographystyle{plain}
	
\end{document}